\documentclass[11pt, oneside]{article}   	
\usepackage{geometry}                		
\usepackage{graphicx}				
\usepackage{amssymb}
\usepackage{amsmath}
\usepackage{bbm}
\usepackage{stmaryrd}
\usepackage{hyperref}
\usepackage{authblk}

\usepackage{amsmath,amssymb,amsthm,tikz,bbm,bm}

\RequirePackage{color}
\definecolor{nicos-red}{rgb}{0.75,0.0,0.0}
\definecolor{light-gray}{gray}{0.6}
\definecolor{really-light-gray}{gray}{0.8}
\definecolor{sussexg}{rgb}{0.0,0.5,0.5}
\definecolor{sussexp}{rgb}{0.5,0.0,0.5}

\usepackage{nicefrac}

\newtheorem{theorem}{\sc Theorem}[section]
\newtheorem{lemma}[theorem]{\sc Lemma}

\numberwithin{equation}{section}
\theoremstyle{remark}

\newcommand{\be}{\begin{equation}}
\newcommand{\ee}{\end{equation}}

\def\bE{\mathbb{E}}
\def\bN{\mathbb{N}}
\def\bP{\mathbb{P}}

\def\N{\bN}

\def\E{\bE}
\def\P{\bP} 

\definecolor{darkgreen}{rgb}{0.0,0.5,0.0}
\definecolor{darkblue}{rgb}{0.0,0.0,0.3}
\definecolor{nicosred}{rgb}{0.65,0.1,0.1}
\definecolor{light-gray}{gray}{0.7}
\allowdisplaybreaks[1]

\RequirePackage{datetime} 

\usepackage{mathtools}
\DeclarePairedDelimiter\ceil{\lceil}{\rceil}
\DeclarePairedDelimiter\floor{\lfloor}{\rfloor}

\title{A new method for the robust characterisation of pairwise statistical dependency between point processes}
\author[1]{Antoine Messager}
\author[2]{Nicos Georgiou}
\author[1]{Luc Berthouze}			
\affil[1]{Department of Informatics, University of Sussex}
\affil[2]{Department of Mathematics, University of Sussex}
\date{}                     
\begin{document}
\maketitle

\begin{abstract}

\noindent The robust detection of statistical dependencies between the components of a complex system is a key step in gaining a network-based understanding of the system. Because of their simplicity and low computation cost, pairwise statistics are commonly used in a variety of fields. Those approaches, however, typically suffer from one or more limitations such as lack of confidence intervals requiring reliance on surrogate data, sensitivity to binning, sparsity of the signals, or short duration of the records. In this paper we develop a method for assessing pairwise dependencies in point processes that overcomes these challenges. Given two point processes $X$ and $Y$ each emitting a given number of events $m$ and $n$ in a fixed period of time $T$, we derive exact analytical expressions for the expected value and standard deviation of the number of pairs events $X_i,Y_j$ separated by a delay of less than $\tau$ one should expect to observe if $X$ and $Y$ were i.i.d. uniform random variables. We prove that this statistic is normally distributed in the limit of large $T$, which enables the definition of a Z-score characterising the likelihood of the observed number of coincident events happening by chance. We numerically confirm the analytical results and show that the property of normality is robust in a wide range of experimental conditions. We then experimentally demonstrate the predictive power of the method using a noisy version of the common shock model. Our results show that our approach has excellent behaviour even in scenarios with low event density and/or when the recordings are short. 
\end{abstract}

\section{Introduction}

\noindent Network-based modelling, whereby nodes denote components of the system and edges denote interactions between those components, has become a paradigm of choice for describing, understanding and controlling complex systems~\cite{newman2018networks}. It is widely used in many fields including neuroscience~\cite{sporns2010networks}, mathematical epidemiology~\cite{newman_spread_2002} and social sciences~\cite{kwak2010twitter} to name just a few. However, in many cases, the actual connectivity between individuals may not be directly available or may be incomplete. A substantial body of work has therefore focused on inferring links using the activity of the nodes (e.g. neural spikes, computer events, tweets, gene expression levels) and coupling measures~\cite{blinowska2011review, quaglio2018methods}. In this paper, we focus on the problem of inferring \textit{functional connectivity}, i.e., the graph of pairwise statistical dependences between components of the system as opposed to \textit{causal connectivity} where such dependences are assumed to be causal~\cite{rubinov2010complex}. Further, we restrict ourselves to systems involving point processes. Although the discrete nature of such processes provides greater analytical tractability, reliable inference of functional coupling between components remains a challenging problem: the exact nature of the point processes may not be known~\cite{rummel2010analyzing}; non linearity, non stationarity and noise are often to be expected~\cite{muller2008estimating, netoff2004analytical, pascual2007instantaneous}; and finally, recordings may be short and repeated measurement unavailable \cite{rummel2010analyzing}.\newline

\noindent A common way to quantify functional interaction between two point processes consists in computing Pearson's cross-correlation and selecting its maximal value~\cite{kramer2009network, zandvakili2015coordinated}. Theoretically, after a hyperbolic transformation, the resulting value can then be compared to the theoretical value of the standard deviation and an interpretable Z-score may be extracted~\cite{fisher1921probable}. However, there are a number of issues with this approach: it assumes stationarity\footnote{The issue of stationarity is common to most methods (including our own) in their basic form and is typically addressed through windowing and/or more sophisticated filtering methods.}, it is bin dependent, it assumes that recordings are infinitely long, the Z-score is only approximately normal if the processes are bivariate normally distributed (hence making its interpretation somewhat hazardous) and the selection of the maximal value limits the ability to characterise different patterns of interaction, especially if a large range of lags is considered. Hence, in neuroscience and genetics, in particular, it is still very common to simply threshold the cross-correlation, where the threshold is arbitrarily chosen~\cite{buckner2009cortical,eguiluz2005scale} or based on some expected or desired~\cite{bassett_adaptive_2006,hayasaka_comparison_2010,wang2009parcellation} characteristics of the resulting network. Whilst such an approach removes the assumption of bivariate normal distribution, it is somewhat unsatisfactory because in general the characteristics of the  network being inferred are not known. More recent studies rely on comparing the cross-correlation to that of surrogate data. For example, Smith et al.~\cite{smith2011network} create null data consisting of testing timeseries from different subjects, i.e., without causal connections between them. This enables the choice of a threshold and interpretation of the output. However, this does not solve the other aforementioned issues and also introduces the computational cost of generating the surrogate data. \newline

\noindent A powerful alternative to the use of cross-correlation to infer functional connectivity is the frequency-domain concept of coherence (along with its somewhat less intuitive time-domain counterpart, the cumulant), applicable to both time series and point processes (or both) and for which rigorous confidence intervals have been derived~\cite{halliday1995framework}. Whilst this framework addresses many of the aforementioned issues (and we note recent work extending it to characterising directionality of interaction~\cite{halliday_nonparametric_2015}), its reliance on long recordings can make analysis problematic for short recordings or when the data are sparse~\cite{jarvis2001sampling}. The issue of short and/or sparse data was addressed in much prior work when some sought to count the expected number of coincidental pair of events and derive its expected standard deviation~\cite{johnson1976analysis}. This allowed the derivation of an interpretable Z-score and was successfully used to separate purely stimulus-induced correlations from intrinsic interneuronal correlation~\cite{shao1987normalized, eggermont1992neural, arnett1981cross}. Similar results were independently found by Palm et al. who thoroughly analysed them~\cite{palm1988significance}. Still, this approach did not address other issues such as bin dependency or the choice of a specific lag. \newline

\noindent In recent years, the neuroscience community (in particular) has sought to move away from pairwise statistics and reveal higher order structures~\cite{quaglio2018methods, brown2004multiple}. And whilst some of those newer approaches do meet the aforementionned challenges and efficiently recover functional networks (see~\cite{chen2011statistical} for example), application to large systems can render them slow and hence not suitable to large networks such as those found in commercial computer networks~\cite{annet-messager} or social media networks. In addition, these Bayesian approaches typically see their performance decrease with the size of the network due to the increase in the number of parameters to be fitted~\cite{stevenson2009bayesian}. \newline

\noindent In this paper, we derive a pairwise statistic of interaction between point processes that is computationally inexpensive and that overcomes all aforementioned problems: it is bin independent, it measures interactions at all lags, it does not require that the recording be infinitely long, it comes with confidence intervals. In Section~\ref{sec:method} we first derive exact analytical expressions for the expected value and standard deviation of the statistic. We then prove that this statistic is normally distributed. In Section~\ref{sub:agreement}, we demonstrate the excellent agreement between analytical and experimental results. In Section~\ref{sub:convergence}, we first experimentally confirm that the statistic converges to a normal distribution (as proven previously) and therefore enables the construction of a Z-score. We describe the role of rates and lag in the time needed for convergence. In Section~\ref{sub:sens_anal}, we use the delayed common shock model to demonstrate the effectiveness of our statistic in characterising both instantaneous and delayed interaction.

\section{Methods}
\label{sec:method}
\subsection{Graphical construction of the theoretical model}

For any integer time horizon $T$ we construct two independent Bernoulli sequences, $\widetilde{\bf X}_T$ and $\widetilde{\bf Y}_T$ up to $T$. To be precise, 
$\widetilde{\bf X}_T = (X_1, \ldots, X_T)$ and $\widetilde{\bf Y}_T = (Y_1, \ldots, Y_T)$. Each coordinate $X_i$,  $Y_j$ are i.i.d.\ Bernoulli($p_X$) and Bernoulli($p_Y$) respectively, 
and each realization is a sequence of 1's and 0's. Parameters $p_X$ ( resp. $p_Y$) is the probability of seeing a 1 in the $\widetilde{\bf X}_T$ (resp. $\widetilde{\bf Y}_T$) sequence. 
 On average, by the strong law of large numbers, one expects $p_X T$ many 1's in the $\widetilde{\bf X}_T$ sequence and $p_Y T$ in the $\widetilde{\bf Y}_T$ sequence, for large $T$ values. 
From central limit theorem considerations, the actual number of 1's is up to two leading orders $\widetilde{\bf X}_T$ is $p_X T+ c_Z \sqrt{T}$ where $c_Z$ will be a random normally distributed number, as long as $T$ is large enough. \newline 

\noindent We create a graphical arrangement of \emph{marks} using the two independent sequences on $\llbracket1, T\rrbracket^2$. A lattice square $(i,j)$ is considered marked if and only if $X_i = Y_j = 1$. We also define a \emph{lag}  $\delta $,  $(0 \le \delta \le T)$; we are interested in the number of times the two processes obtained the value 1 in a time interval of size $\delta$ in either direction; in other words we want to know how many marks exist in a band of vertical height $2\delta+1$ around the main diagonal $D = \{ (i,i): 1 \le i \le T \}$. \newline

\noindent In general, when there is availability of data, we can count the number of ones in the available time series (say $n_X, n_Y$) and then infer an approximation to the (generally unknown) parameters 
 $p_X \approx n_XT^{-1}$ and $p_Y \approx n_YT^{-1}$. Then one can run the theoretical model, as we describe it above, up to time horizon $T$ and create two new independent time series, compare them with data and use them for predictions. This we use in Section \ref{subsec:normality}, to derive an approximate central limit theorem (CLT) for the number of marks in a band, as $T \to \infty$. Part of the CLT is the expected value of marks of this completely independent model and approximate standard deviation. We use this result to construct a suitable Z-score for the number of marks in the band.  \newline
 
\noindent Consider a given lag $\delta$ and arrangements $X$, $Y$. We will find statistics for the number of marks in the band of distance $\delta$ around the main diagonal. 
The set of marks in the $\delta$-band $D_{T, \delta} = \{ (i,j) \in \llbracket1, T\rrbracket^2: |i -j| \le \delta \}$ is 
\be\label{eq:numm}
S_{T, \delta} =  \{(i, j) \in \llbracket 1, T\rrbracket ^2: |i -j| \le \delta, X_i = Y_j = 1 \}.
\ee
We denote by $|S_{T, \delta}|$ the cardinality of the set $S_{T, \delta} $; this is precisely the number of marks in $D_{T, \delta}$. 
In the next two subsections we show the calculations for the expected value, limiting standard deviation and approximate normality for 
$|S_{T, \delta}|$.\\

\noindent In all calculations that follow we adopt the standard notation of indicator functions. These are (Bernoulli) random variables that take values 1 or 0. For any $\omega$ in a sample space $\Omega$ and any event $A$
we denote by $1\!\!1\{ A \}$ the random variable which satisfies
\[
1\!\!1\{ A \}(\omega) = 
\begin{cases}
1, & \omega \in A,\\
0, & \omega \in \Omega \setminus A.
\end{cases}
\]
The distribution of the indicators is Bernoulli, with probability of success $\P\{A\}$ and therefore the expected value satisfies $\E(1\!\!1\{ A \}) = \P\{A\}$. We use these facts without any particular mention in the sequence.

\subsection{Expected value of the number of points in a $\delta$-band.}
\label{subsec:expvar}

For the purposes of this section, fix parameters $p_X$ and $p_Y$ to denote the probability of success for two independent Bernoulli sequences 
$\{ X_i \}_{ i \ge 1}$ and $\{ Y_j \}_{ j \ge 1}$ so that each sequence is i.i.d.\ with marginal distributions 
\[  X_i \sim \text{\rm Ber}(p_X) \quad \textrm{ and }  Y_j \sim \text{\rm Ber}(p_Y).\]
For any $T \in \N$, the random variable under consideration is  $|S_{T, \delta}|$.\\

\noindent
The computations rely on the decomposition 
\[
|S_{T, \delta} | = \sum_{(i, j): |i-j|\le \delta} 1\!\!1\{  X_i = Y_j = 1 \}. \newline
\]

\noindent With this we compute first the expected value of the number of marks in $D_{T,\delta}$.
\begin{align} 
\label{eq:expected_value}
\E(|S_{T, \delta} |)&=\E\bigg(\sum_{(i, j): |i-j|\le \delta} 1\!\!1\{  X_i = Y_j = 1 \}\bigg) =
\sum_{(i, j): |i-j|\le \delta} \E\left(1\!\!1\{  X_i = Y_j = 1 \}\right) \nonumber \\
&= \sum_{(i, j): |i-j|\le \delta} \P\{  X_i =1,  Y_j = 1  \} = \sum_{(i, j): |i-j|\le \delta} \P\{  X_i =1\}\P\{  Y_j = 1  \} \nonumber \\
&= p_Xp_Y \big( T(2\delta+1)- \delta(\delta+1)\big).
\end{align}
At the last line of the computation above we used the number of terms in the sum. This can be calculated by adding the integer cells on the $2\delta+1$ diagonals,
\be\label{eq:atd}
A_{T, \delta} =  T + 2 \sum_{k=1}^\delta (T-k) = T(2\delta+1)- \delta(\delta+1). 
\ee
Finally, observe that $A_{T, 0} = T$, which is precisely the size of the main diagonal. \newline

\subsection{Central limit theorem for number of marks in $\delta$-band as $T \to \infty$}
\label{subsec:normality}

\noindent We now present a different way to count the marks in the band, that will lend itself into the application of an ergodic CLT. 
For any $\delta+1 \le i \le T - \delta - 1$ consider the vector random variable of dimension $ d = 2\delta+1$ given by
 \[
 {\bf Y_i} = [ Y_{i - \delta}, Y_{i-\delta+1}, \ldots, Y_{i+\delta -1},  Y_{i+\delta}].
 \]
 Then the total number of marks between  $\delta+1 \le i \le T - \delta - 1$ is given by 
 \be
 L_{T, \delta} = \sum_{i=\delta +1}^{T - \delta -1}  X_i ( {\bf Y}_i \cdot{\bf 1}_{d}),
 \ee
 where ${\bf 1}_d$ is the $d$-dimensional vector $(1, 1 \ldots 1)$. For each $i$, a direct calculation gives that 
 \[
 \E( X_i ( {\bf Y}_i \cdot{\bf 1}_{d})) =  \E( X_i )\E( {\bf Y}_i \cdot{\bf 1}_{d})) = p_X p_Y(2\delta+1).
 \]
 Define $W_i = X_i ( {\bf Y}_i \cdot{\bf 1}_{d})$ for notational convenience.\newline
 
 \noindent Before stating the theorem, two observations follow. First we have the immediate inequality   
 \[
  L_{T, \delta}  \le  | S_{T, \delta}| \le L_{T, \delta} + (\delta+1)^2.
 \]
 Thus, as $T \to \infty$,  asymptotically  $L_{T, \delta} \sim |S_{T, \delta}|$.  The same inequality holds for expectations. Scaled by the same quantity,  both variables satisfy the same limiting law as long as $\delta$ does not depend on $T$. If $\delta$ depends on $T$, one needs more refined arguments to show limiting results, and they will depend on this relation as well. \newline
 
\noindent Second, notice that the random numbers $\{ W_i \}_{i}$ have the same distribution for each fixed $i$ and they are stationary and ergodic. Moreover, as long as  $| i - k | > 2\delta + 1$, variables $W_i = X_i ( {\bf Y}_i \cdot{\bf 1}_{d})$ and $W_k = X_k ( {\bf Y}_k \cdot{\bf 1}_{d})$ are completely independent. Therefore we have just defined a stationary ergodic sequence that is $2\delta+1$-dependent only. In particular this implies that the stationary sequence is \emph{strongly mixing}. \newline
 
 \noindent We introduce some notation.
 First, let 
 \[
 Z_i =   X_i ( {\bf Y}_i \cdot{\bf 1}_{d}) - p_X p_Y(2\delta+1) = W_i - \E(W_i).
 \]
 Therefore $\E(Z_i) = 0$. Also note that $\E(Z_i^{\ell}) < \infty$ for any $\ell$ (high moment hypotheses are crucial  for CLT for ergodic sequences). 
 Then define (the value does not depend on the index $\ell$ below as long as $\ell \ge \delta + 1$)
 \be \label{var}
 \sigma_{\delta}^2 = \E(Z_\ell^2) + 2\sum_{k=1}^{2 \delta+1} \E(Z_\ell Z_{\ell+k})= {\rm Var}(W_{\ell}) + 2 \sum_{k=1}^{2 \delta+1} {\rm Cov}(W_\ell,  W_{\ell+k}).
 \ee

\begin{lemma}
The constant $\sigma^2_{\delta}$ in \eqref{var} is given by 
\be\label{sigma2}
\sigma_\delta^2  = (2\delta+1)p_Xp_Y(1 - p_Xp_Y ) + 2\delta(2\delta+1) p_Xp_Y( p_Y(1-p_X) + p_X(1-p_Y)). 
\ee
\end{lemma}

\begin{proof}

We are going to compute each of the terms appearing in the last expression of \eqref{var}. First, 
\begin{align*}
{\rm Var}(W_{\ell}) &= {\rm Cov}(W_\ell,  W_{\ell}) =  {\rm Cov}\Big(\sum_{j = -\delta}^{\delta} X_\ell Y_{\ell + j},  \sum_{k = -\delta}^{\delta} X_\ell Y_{\ell + k}\Big) \\
&= \sum_{j = -\delta}^{\delta} \sum_{k= -\delta}^{\delta}  {\rm Cov}( X_\ell Y_{\ell + j},  X_\ell Y_{\ell + k}) = \sum_{j = -\delta}^{\delta} \sum_{k= -\delta}^{\delta} \left( \E( X^2_\ell Y_{\ell + j} Y_{\ell + k}) - \E(X_\ell Y_{\ell + j})\E( X_\ell Y_{\ell + k})\right)\\
&= \sum_{i= -\delta}^{\delta} \left( \E( X^2_\ell Y_{\ell + i} ^2) - p_X^2p_Y^2 \right)+ \sum_{j = -\delta}^{\delta} \sum_{k= -\delta, k \neq j}^{\delta} \left( \E( X^2_\ell Y_{\ell + j} Y_{\ell + k}) - p_X^2p_Y^2\right)\\
&= \sum_{i= -\delta}^{\delta} \left( p_Xp_Y - p_X^2p_Y^2 \right)+ \sum_{j = -\delta}^{\delta} \sum_{k= -\delta, k \neq j}^{\delta} \left( \E( X^2_\ell Y_{\ell + j} Y_{\ell + k}) - p_X^2p_Y^2\right)\\
&= \sum_{i= -\delta}^{\delta} \left( p_Xp_Y - p_X^2p_Y^2 \right)+ \sum_{j = -\delta}^{\delta} \sum_{k= -\delta, k \neq j}^{\delta} \left( p_Xp_Y^2 - p_X^2p_Y^2\right)\\
&= (2\delta+1)p_Xp_Y(1 - p_Xp_Y + 2\delta p_Y(1-p_X)).
\end{align*}
The covariance for the other terms is computed in a similar way. First note that the variables common between $W_{\ell}$ and $W_{\ell + k}$ are the 
$Y_{\ell + j}$ for which $-\delta + k \le j \le \delta$ which precisely equal the values for $Y_{\ell + k + i }$ when $-\delta \le i \le \delta- k$. Then
\begin{align*}
{\rm Cov}(W_\ell,  W_{\ell+k}) &=  {\rm Cov}\Big(\sum_{j = -\delta}^{\delta} X_\ell Y_{\ell + j},  \sum_{i = -\delta}^{\delta} X_{\ell+k} Y_{\ell + k +i}\Big)=  {\rm Cov}\Big(\sum_{j = -\delta+k}^{\delta} X_\ell Y_{\ell + j},  \sum_{i = -\delta}^{\delta-k} X_{\ell+k} Y_{\ell + k +i}\Big)\\
&=  {\rm Cov}\Big(\sum_{j = -\delta+k}^{\delta} X_\ell Y_{\ell + j},  \sum_{i = -\delta+k}^{\delta} X_{\ell+k} Y_{\ell + i}\Big) = \sum_{j = -\delta+k}^{\delta} {\rm Cov}\Big( X_\ell Y_{\ell + j}, X_{\ell+k} Y_{\ell + j}\Big) \\
&=  \sum_{j = -\delta+k}^{\delta}\E(X_{\ell}) \E(X_{\ell +k})\E(Y_{\ell+j}^2) -  \E(X_{\ell}) \E(X_{\ell +k})\E(Y_{\ell+j})^2 \\
&= (2\delta - k +1) p_X^2 p_Y(1-p_Y).
\end{align*} 
Then, overall
\begin{align*}
 \sigma_{\delta}^2&=  (2\delta+1)p_Xp_Y(1 - p_Xp_Y + 2\delta p_Y(1-p_X)) + 2  p_X^2 p_Y(1-p_Y) \sum_{k=1}^{2\delta+1}  (2\delta - k +1)\\
  &=  (2\delta+1)p_Xp_Y(1 - p_Xp_Y ) + 2\delta(2\delta+1) p_Xp_Y( p_Y(1-p_X) + p_X(1-p_Y)). \qedhere
\end{align*}

\end{proof}

\noindent Then we can directly apply the central limit theorem for strongly mixing variables and obtain the following
 
\begin{theorem} Let $S_{T, \delta}$ be the number of marks as defined by \eqref{eq:numm}. Then, (with the notation introduced above), the following limit holds in distribution 
\be
\lim_{T \to \infty}  \frac{| S_{T, \delta}| - \E(| S_{T, \delta}|)}{\sqrt{T}} \stackrel{\mathcal D}{=} Z \sim \mathcal N(0, \sigma^2_{\delta}),
\ee
where $\sigma^2_{\delta}$ is given by \eqref{sigma2}.
\end{theorem}
 
\begin{proof}
There exists a constant $C = C(\delta, p_X, p_Y) < \infty$ so that  
\[
 \Big| \frac{| S_{T, \delta}| - \E(|  S_{T, \delta}|)}{\sqrt{T}} - \frac{L_{T, \delta} - \E(L_{T, \delta})}{\sqrt{T}} \Big|\le  \frac{C}{\sqrt{T}}.
\]
This is a $\P$-a.s. statement, so as long as $T \to \infty$ the two fractions have the same distributional limit. Now focus on 
\begin{align*}
\frac{L_{T, \delta} - \E(L_{T, \delta})}{\sqrt{T}} &= \frac{ \sum_{\delta+1}^{T-\delta-1}W_i - (T - 2\delta-2)\E(W)}{\sqrt{T}} = \frac{\sum_{\delta+1}^{T-\delta-1}Z_i}{\sqrt{T}}\\
&= \frac{\sqrt{T-2\delta -1}}{\sqrt{T}}\cdot \frac{\sum_{\delta+1}^{T-\delta-1}Z_i}{\sqrt{T-2\delta-1}}.
\end{align*}
As $T \to \infty$, the fraction multiplying the variable tends to 1 while the second fraction converges weakly to $\mathcal N(0, \sigma^2_{\delta})$ by the CLT for strongly mixing ergodic sequences. 
\end{proof}

\section{Results}

\subsection{Agreement between analytical and empirical estimates}
\label{sub:agreement}

\noindent Given the strong dependence of the analytical expressions on the product of the rates ($p_Xp_Y$), we examined the agreement between analytical and empirical estimates for i.i.d. uniform point processes along two dimensions -- event density and rate homogeneity -- considering both average and limit cases. Concretely, we considered 4 pairs of point processes emitting events in $\llbracket 1;T\rrbracket$, T=1,000, with rates (0.05, 0.05), (0.05, 0.75), (0.75, 0.75) and (0.25, 0.50), respectively. Note that $p_X = p_Y = 0.75$ was chosen because it is the limit (for large lags) of the rate at which variance is maximised in the homogeneous case. As shown by Figure~\ref{fig:ex_accc}, agreement between analytical and empirical estimates is excellent across all scenarios, even when the expected number of events is extremely low. It should be noted that whereas the expected value of the statistic is insensitive to rate inhomogeneity, i.e., the expected value for one pair of point processes emitting (1,100) events and that of a pair emitting (10,10) events will be identical, the variance is not. For a given product of rates, the more heterogeneous the rates, the higher the standard deviation. \newline

\begin{figure}[h]
\includegraphics[width=1.\linewidth]{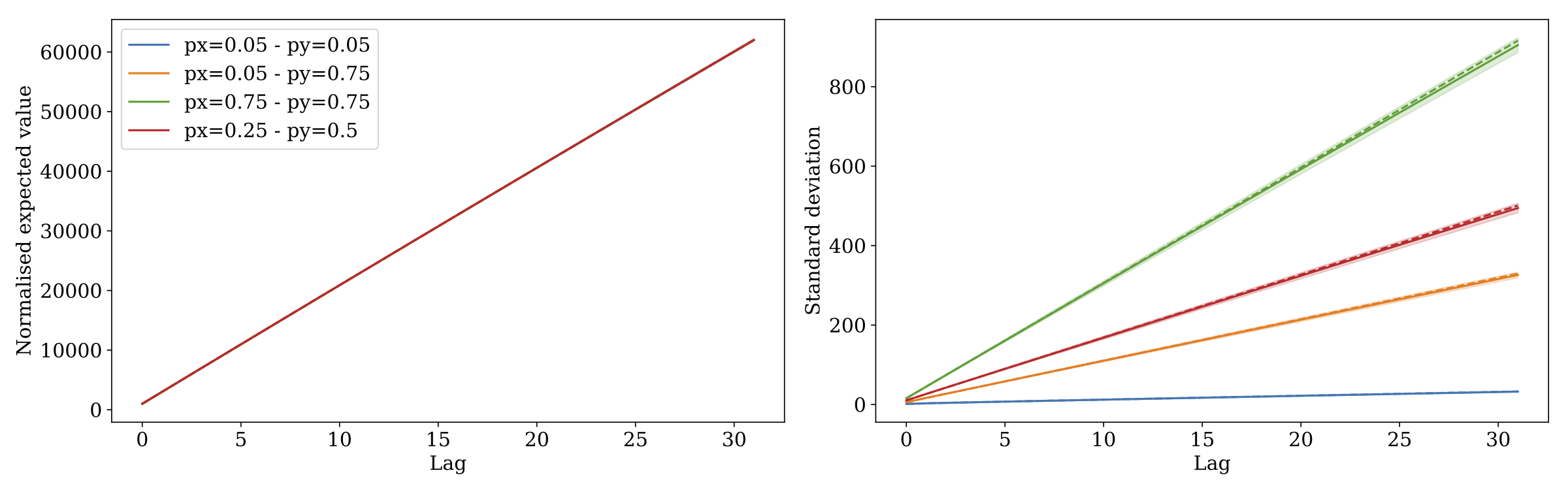}
\caption{Analytical (solid lines) and empirical (dotted lines) expected value (left) and standard deviation (right) as a function of the lag for the four scenarios given in the legend. For ease of presentation, the expected value (left panel) was normalised by $n_X n_Y$. As expected from Eq.~\ref{eq:expected_value}, this leads to all analytical estimates of the expected value being identical. The standard deviations (right panel) are calculated as $\sigma_\delta \sqrt{T}$. In both panels, error bars were computed over 100 sets of 1,000 pairs of point processes. } 
\label{fig:ex_accc}
\end{figure}

\noindent Next, we investigated the extent to which agreement between analytical and empirical estimates holds when considering considering scenarios that depart from the assumptions underlying our analytical derivation. First, we considered the case of sampled (or binned) point processes, as is frequently the case in many real-world scenarios, e.g., in physiology. Next, we introduced auto-correlations in the point processes as this has been shown to induce biases when measuring statistical dependence between point processes, e.g., when considering Poisson processes~\cite{pipa2013impact}.

\subsubsection{Impact of sampling frequency}
\label{sub:sampledPP}
We investigated the impact of sub-sampling both in independent Bernouilli sequences and in homogeneous Poisson processes. First we considered the case of independent Bernouilli sequences. Starting from reference signals of length $T$ (i.e., $T$ time bins of length 1), we constructed sub-sampled sequences of length $\ceil{\frac{T}{L}}$ where each time bin contained a $1$ if there was one or more event in the corresponding $L$ time bins in the original sequences, $0$ otherwise. Note that where $L$ was not a divisor of $T$, the last time bin was smaller than $L$. The effect of this was considered negligible. Sub-sampling has the effect of reducing the effective rate of each sequence. Since each time bin remains independent, the new rate can be calculated as $1-(1-p)^L$ and therefore the analytical formulas we derived for expected value and standard deviation should hold with those rates.

\begin{figure}[h] 
\includegraphics[width=1.0\linewidth]{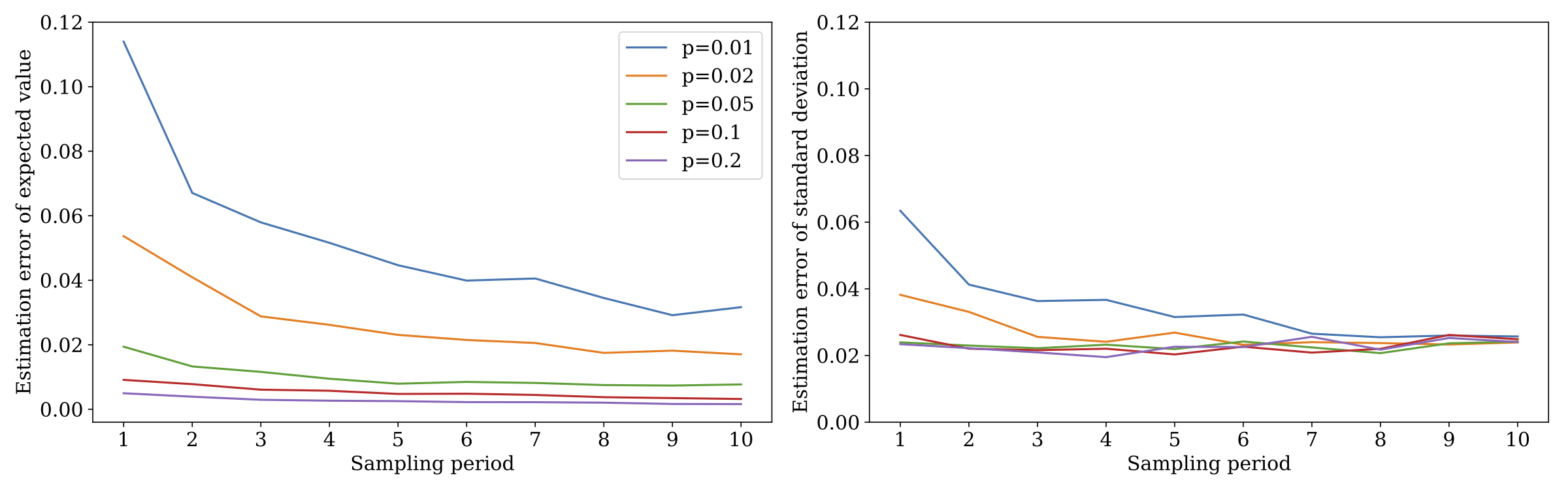}
\includegraphics[width=1.0\linewidth]{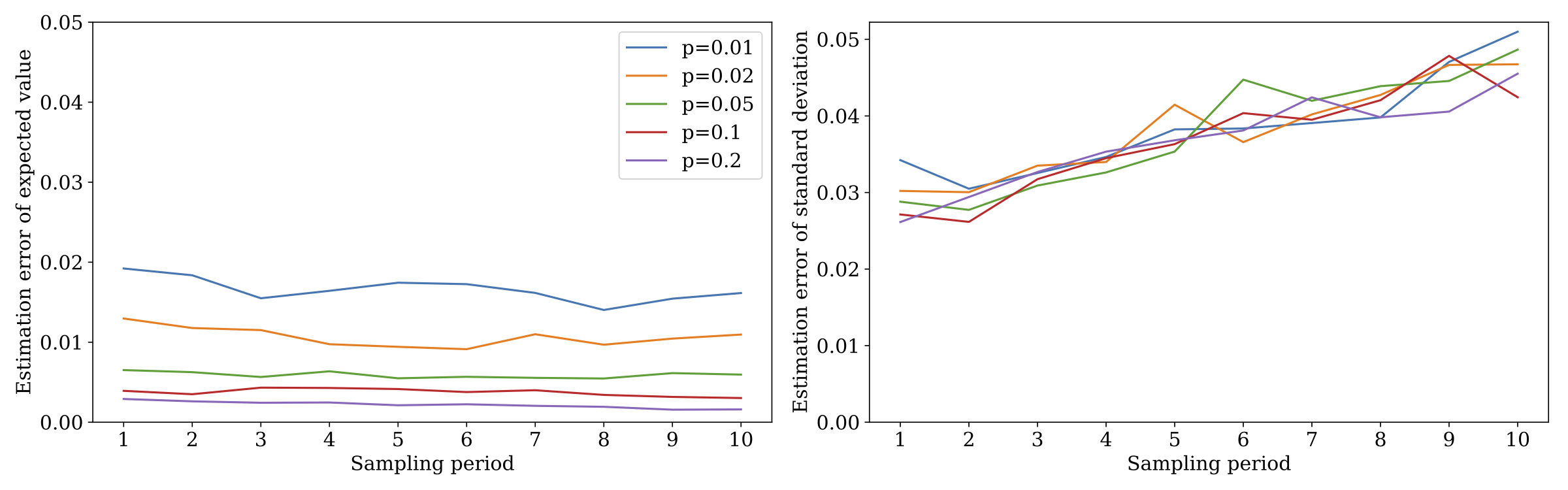}
\caption{Normalised root-mean-square error as a function of the sampling period $L$ for both expected value (left) and standard deviation (right) at lag 0 (top row) and lag $\sqrt{\frac{T}{L}}$ for Bernouilli sequences with rates provided in the legend. One hundred sets of 1,000 pairs of sequences were used. The non-sampled sequences had length $T=1000$.}
\label{fig:bin_size}
\end{figure}

This is confirmed by Figure~\ref{fig:bin_size} which shows small normalised root-mean-square errors (NRMSE) for both quantities, irrespective of the choice of lag and sampling period. For both small (top row) and large (bottom row) lags, estimation of the expected value is shown to be sensitive to the rate $p$, improving with higher rate. This is expected because given the finite (fairly small, in this case) length of the sequence, the smaller the rate, the less robust the estimation is. For the same reason, it is observed that at low lag (here, 0, i.e., instantaneous cross-correlation; top row, left panel), the error decreases with increasing sampling period, with  this improvement increasingly pronounced for decreasing rates. This is because with larger bin sizes, the resulting rate $1-(1-p)^L$ increases, leading to more robust estimations. For sufficiently high rates and sampling period (e.g., $p=0.2$, sampling period of 3; same panel), the error becomes insensitive to the choice of sampling period. Qualitatively similar observations can be made for the estimation of the standard deviation (top row, right panel). For large lags (bottom row), the estimation error of the expected value is insensitive to the choice of sampling period as expected from the fact that at large lags, the number of events allows for a more robust estimate. Further, it is significantly reduced compared to that at low lag (e.g., error for the smallest rate below 0.02 at all sampling periods, compared to up to 0.12 at small lag; bottom row, left panel). The estimation error on the standard deviation is insensitive to the value of the rate but marginally increases with the sampling period (bottom row, right panel). This is expected from the fact that the rates on the basis of which the standard deviation is calculated are an increasingly poorer approximation of the actual rates following sampling.

Next, we considered the case of homogeneous Poisson processes in lieu of Bernouilli sequences. For a given time horizon and rates, homogeneous Poisson processes will emit events that satisfy our assumptions provided the sampling frequency is large enough and therefore we should expect excellent agreement between analytical and empirical estimates. However, if the sampling frequency is reduced (or put differently, the rates increase in relation to the sampling frequency), then the likelihood that two events fall within the same time bin increases. For a Poisson process of rate $\lambda$ the probability that there are $n$ events within any bin of duration $b$ is:
\begin{equation}
P(N(t+b)-N(t) = n)= \frac{(\lambda b)^n}{n!} e^{-\lambda b}
\end{equation}
\noindent And therefore, the expected number of events 'lost' to binning over the given time horizon $T$ is:
\begin{equation}
\label{eq:nb_loss}
\begin{split}
\E(N_L) = \frac{T}{b}\sum_{n=2}^{\infty}(n-1)\frac{(\lambda b)^n}{n!} e^{-\lambda b} 
= \lambda T - \frac{T}{b}(1 - e^{-\lambda b}) \\
\end{split}
\end{equation}

\noindent This shows that for a given time horizon $T$ and known expected number of events, our statistic should be more accurate if the sampling frequency is high (i.e., the bin size $b$ is small). To briefly illustrate this, we investigated agreement when the sampling period was kept constant (unit time bins) but the rates were increased such that the fraction of events lost to binning varied between $0.5\%$ and $20\%$. For the sake of brevity, we once again only considered processes with homogeneous rates and quantified agreement using NRMSE at lags $0$ and $\sqrt{T}$.\newline

As shown by Figure~\ref{fig:poisson_bin_size}(top), when not correcting the rates for the loss of events through binning, both estimation errors increase with larger rates, i.e., as the proportion of events lost to binning increases; and this degradation is rapid, with a $20\%$ loss of events resulting in a NRMSE of almost $40\%$ in the expected value. This error is insensitive to the lag considered, except at very low rates when a low number of events (at lag 0) leads to an unreliable estimation process as explained earlier. When correcting the rates for the loss of events, however, Figure~\ref{fig:poisson_bin_size}(bottom) shows that except for very low rates, there is very good agreement across all rates, as demonstrated in Appendix~\ref{app:binning}. 

\begin{figure}
\center \includegraphics[width=1.0\linewidth]{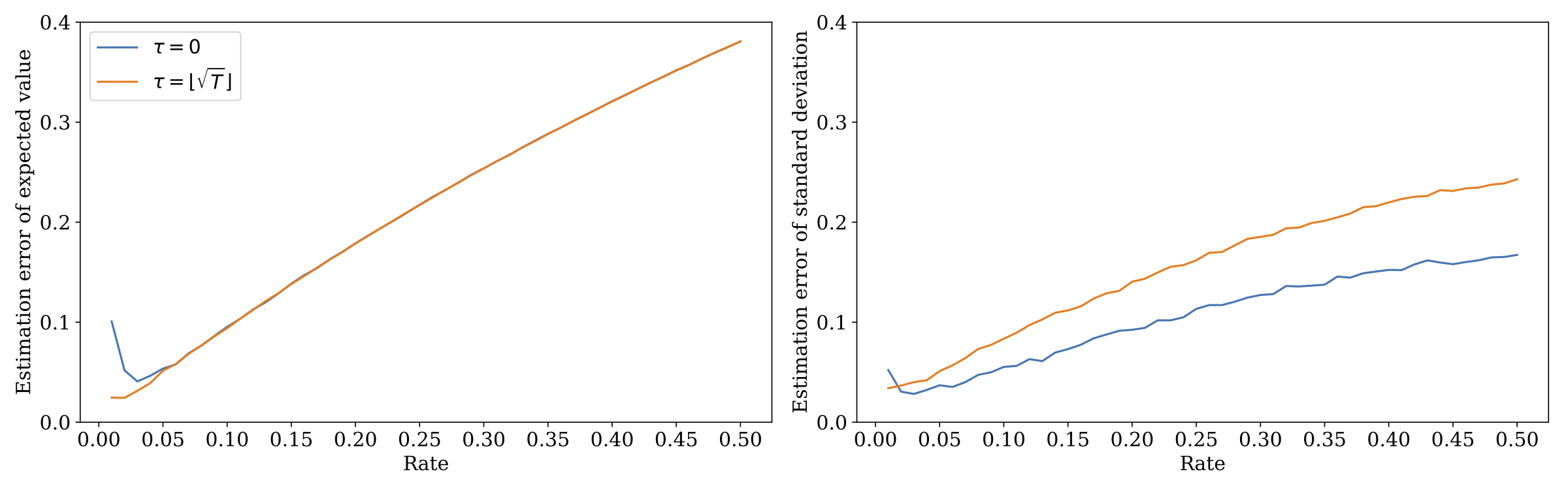}
\includegraphics[width=1.0\linewidth]{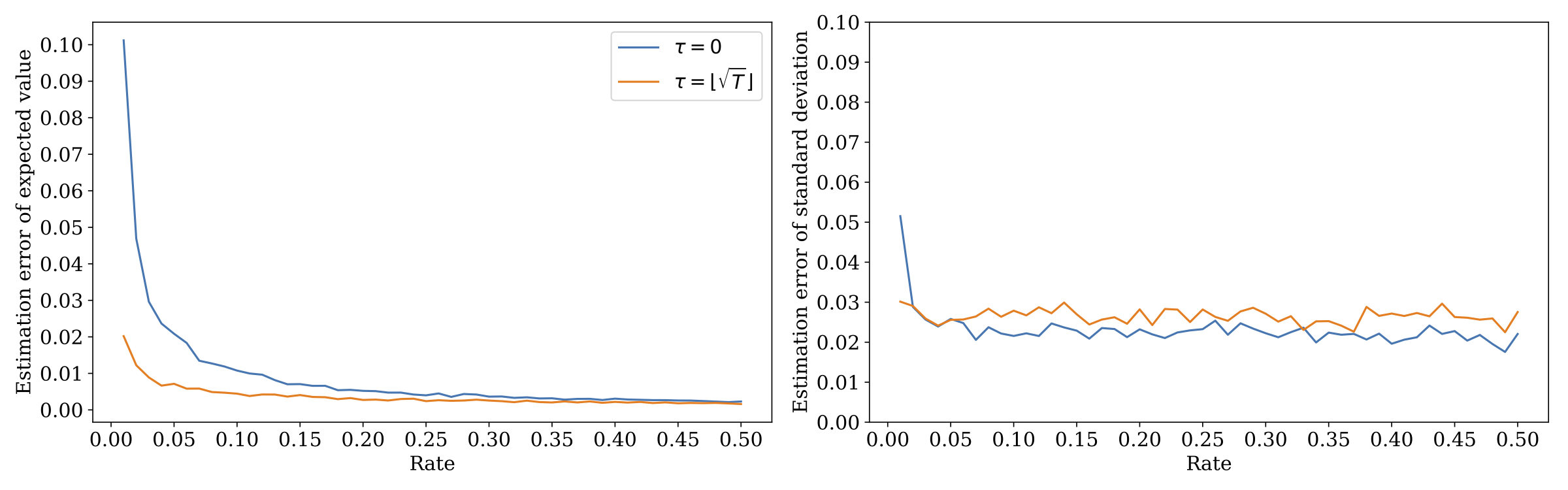}
\caption{RMSE for both expected value (left) and standard deviation (right) at lags 0 and $\sqrt{T}$ for homogeneous Poisson processes with varying rates. Top: Error when the analytical value is calculated assuming no events were lost through binning. Bottom: Error when the rate was estimated based on the observed number of events. One hundred sets of 1,000 pairs of sequences were used.}
\label{fig:poisson_bin_size}
\end{figure}


\subsubsection{Auto-correlation}
We used geometric AR(1) processes~\cite{mckenzie_simple_1985} to generate auto-correlated waiting times for the events of the point processes. The geometric AR(1) process is defined as: 
\begin{equation}
X_t = \alpha * X_{t-1} + B_t G_t
\end{equation}
where $X$ is a discrete random variable Geometric with parameter $0<\theta<1$, $\left\{B_t\right\}$ and $\left\{G_t\right\}$ are independent sequences of iid random variables such that the random variable $B_t$ is Binary with parameter $0\leq\alpha\leq1$ and the random variable $G_t$ is Geometric with parameter $\theta$, and the $*$ operator is defined as:
\begin{equation}
\alpha * X_t  = \sum_{i=0}^{X_t}Y_i \\
\end{equation}
where $\left\{Y_i\right\}$ is a sequence of iid Binary random variables with parameter $\alpha$. Put simply, this model states that the elements at time $t$, $X_t$, are the sum of (1) the survivors at time $t-1$, $X_{t-1}$, each with probability of survival $\alpha$ and of (2) elements from the innovation process $B_tG_t$. For any non negative integer $k$, the auto-correlation at lag $k$ is given by $\rho(k) = \alpha^k$. In particular, $\alpha$ is the auto-correlation at lag 1.\\

\noindent To quantify the effect of auto-correlations on the agreement between analytical and empirical estimates, we generated point processes systematically varying $\alpha$ for processes with rate $p \in \left\{0.01,0.02,0.05,0.1,0.2\right\}$ over $\llbracket 1;T\rrbracket$, T=1,000. Agreement was once again quantified in terms of the normalised root-mean-square error. As shown by Figure~\ref{fig:autocor}(top row), the estimation error in both expected value and standard deviation increases with the degree of auto-correlation, as expected from the fact that the data increasingly depart from the assumption of independence underlying the derivation of our statistic. Strikingly, we observe that for all lags, the estimation errors on increase as the rates decrease. This can be explained as follows. For a fixed horizon T, the larger the rate, the more constrained the distribution of inter-event intervals is. Auto-correlations introduce structure in the sequencing of the inter-event intervals which exacerbates the covariance term. As a result, empirical estimates show less variability around the mean. For low levels of correlation ($<0.2$) and sufficiently high rates ($>0.05$), there is good agreement between analytical and empirical estimates of the expected value. In fact, as shown by the bottom row of Figure~\ref{fig:autocor}, for $\alpha=0.1$, there is very good agreement for both expected value and standard deviation over all lags even in the worst case scenario of $p=0.01$ (where the expected number of events is only 10). 
\newline

\begin{figure}[!]
\includegraphics[width=1.0\linewidth]{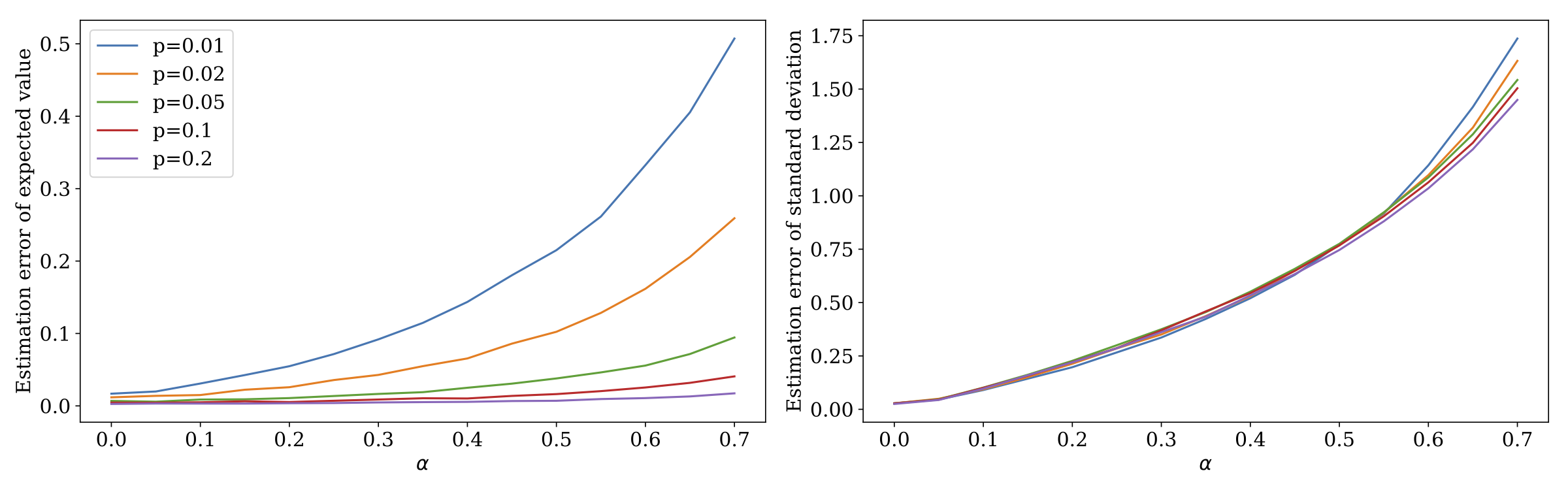}
\includegraphics[width=1.0\linewidth]{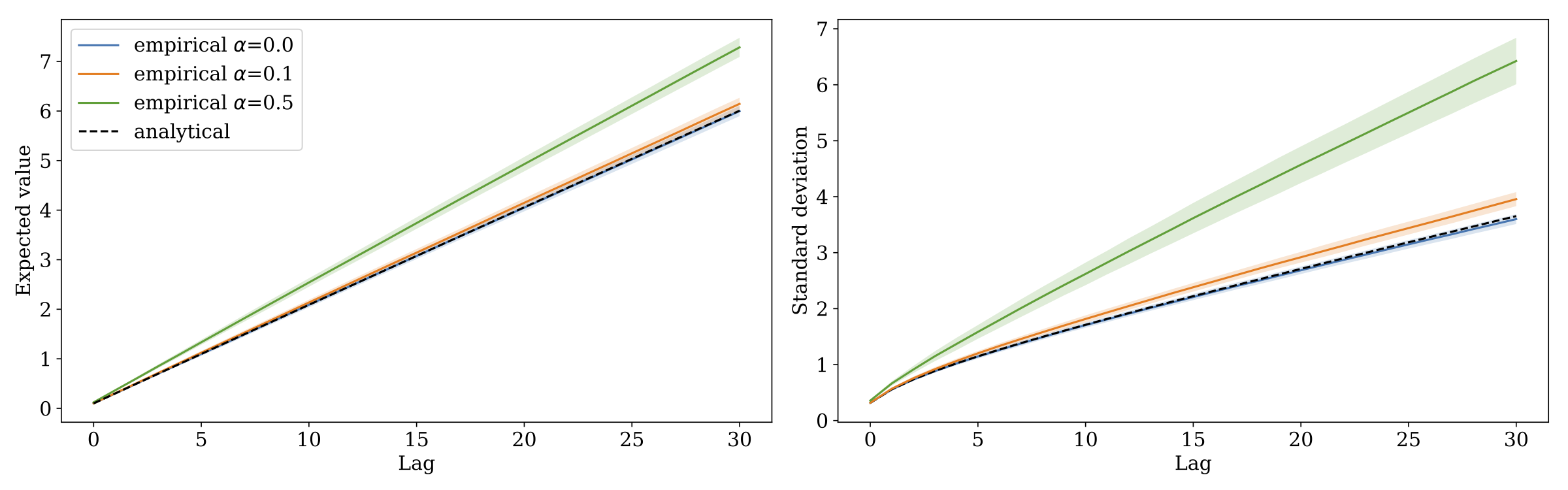}
\caption{Top row: Normalised root-mean-square estimation error of the expected value (left panel) and standard deviation (right panel) as a function of $\alpha$ (the auto-correlation at lag 1) at lag $\floor{\sqrt{T}}$ for the 5 rates shown in the legend. Those errors were calculated over over 100,000 estimates. Bottom row: Empirical expected value (left panel) and standard deviation (right panel) as a function of the lag for $p=0.01$ (the worst case scenario) and three values of auto-correlation $\alpha\in\left\{0,0.1,0.5\right\}$. The analytical curves (black dotted line) are provided for reference. Shaded areas denote the standard deviation of the empirical estimates and were computed over one hundred sets of 1,000 estimates.}
\label{fig:autocor}
\end{figure}

\subsection{Normal convergence of the empirical estimates}
\label{sub:convergence}

\noindent In Section~\ref{subsec:normality}, we have shown that in the limit of large horizon times T, the distribution of our statistic converges to a normal distribution with parameters the analytical results provided in \eqref{eq:expected_value} for the expected value and \eqref{sigma2} for the variance. By way of empirical validation, we systematically varied rates and lags and showed that there was always a data duration after which the distribution of the empirical estimates could be said to be normally distributed based on the evidence of a KS test\footnote{Normality was also assessed using both Shapiro Wilk and Anderson Darling tests and results were consistent across all three tests (data not shown).} with $p > .05$ (with 5000 estimates). Figure~\ref{fig:normality_heat_map} shows those times for the rates and lags considered. Since the expected value is only dependent on the product of the rates, we used homogeneous rates, $\lambda_x = \lambda_y = \frac{n}{T} \in \{0.01,0.02,...,0.5\}$. Focusing on homogeneous rates is justified by the fact that convergence is only marginally affected by heterogeneity, as demonstrated by panel D. The gradual convergence of the empirical distributions to a normal distribution is illustrated in panels B and C when the lag and recording duration increase, all other factors being kept constant. \newline

\begin{figure}[p]
\center\includegraphics[width=0.5\linewidth]{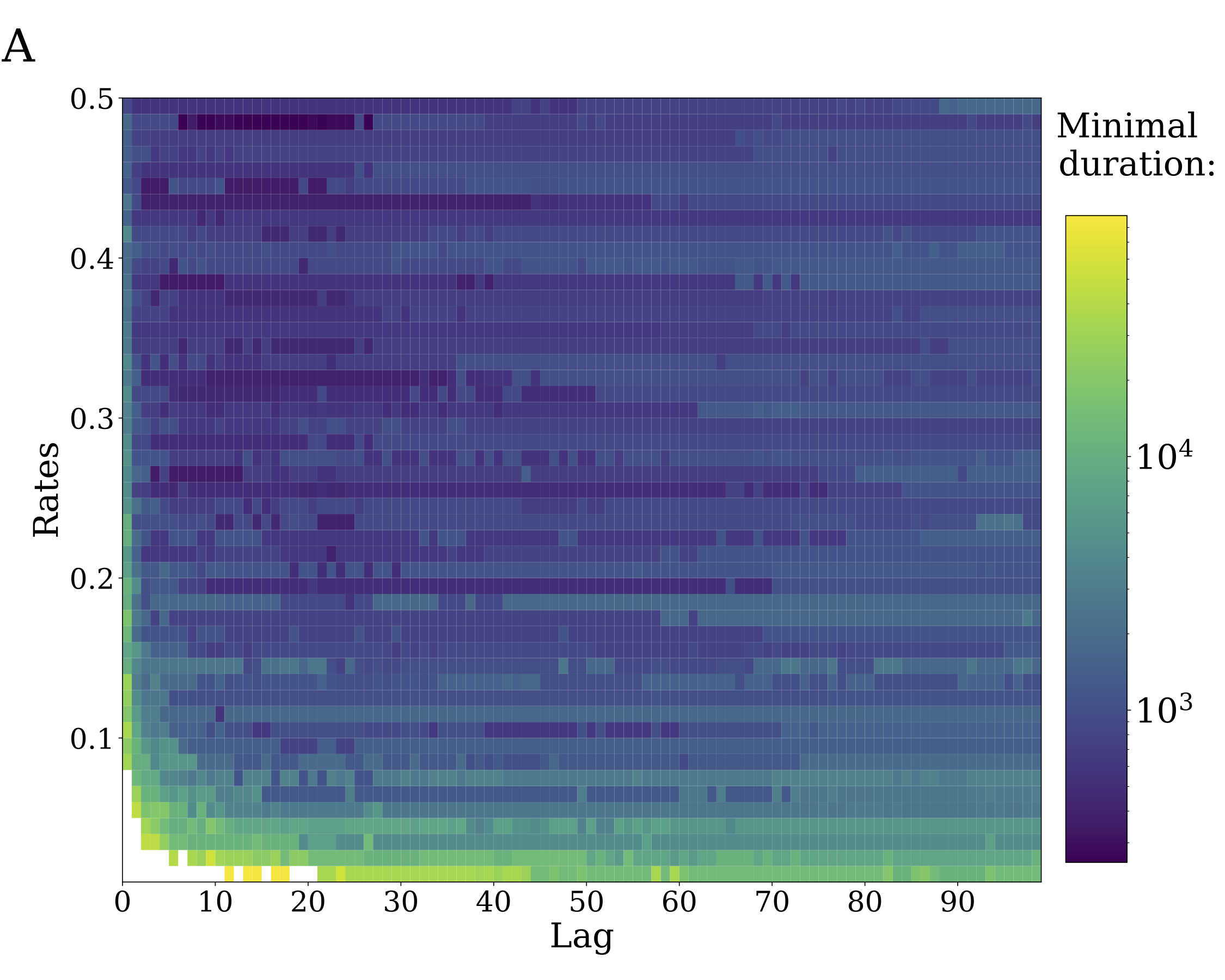}
\includegraphics[width=0.95\linewidth]{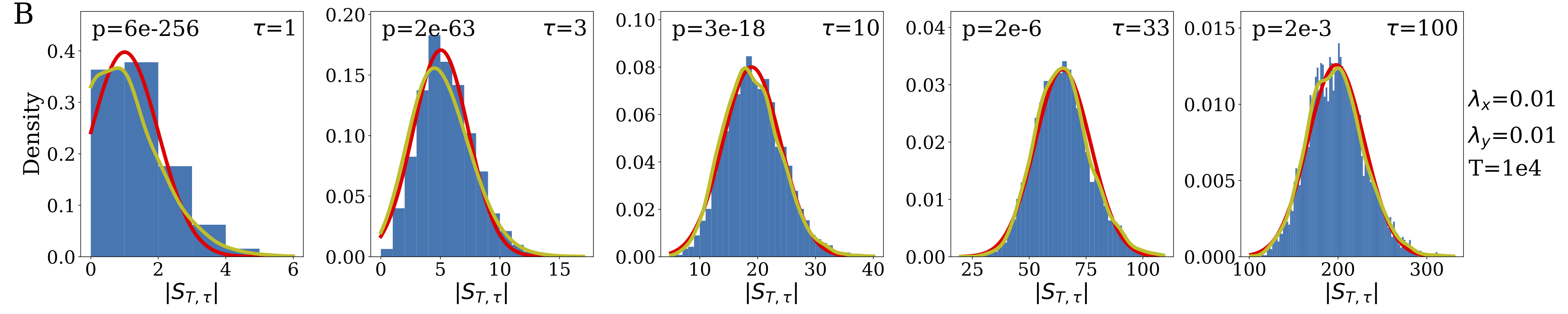}
\includegraphics[width=0.95\linewidth]{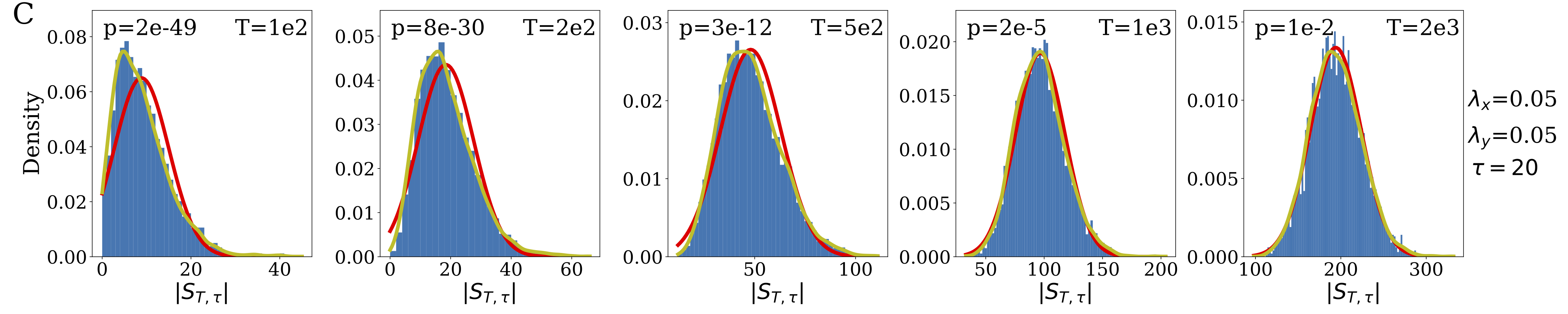}
\includegraphics[width=0.95\linewidth]{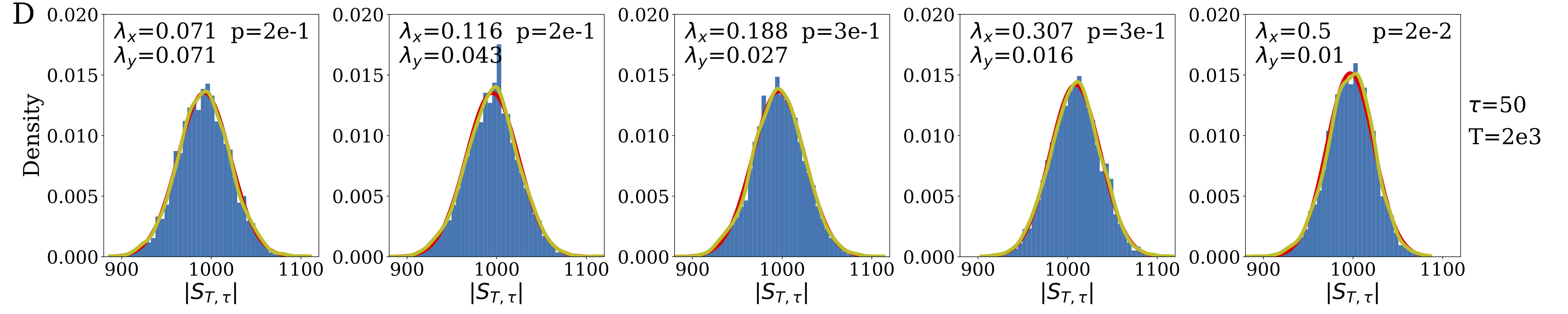}
\caption{\textbf{A:} Minimal recording duration after which the empirical distribution is approximately normal as a function of rates and lags. \textbf{B:} Empirical distributions as a function of the lag for given total duration and rates. \textbf{C:} Empirical distributions as a function of the total duration for given lag and rates. \textbf{D:} Empirical distributions as a function of rate heterogeneity for given total duration, lags and rate products. All panels: (red) the theoretical distribution for the scenario; (yellow) the kernel density estimation of the empirical distribution.} 
\label{fig:normality_heat_map}
\end{figure}

\noindent Empirical convergence to a normal distribution is significantly impacted by the fact that our statistic takes integer values (it is a count of co-occurrent events). First, normal approximation can only be achieved if there is a wide range of values the statistic can take. This is clearly evidenced by panels B and C whereby the p-value increases (the distribution becomes more normal) as recording duration (respectively lag) increases leading to a wider range of values. Using kernel density estimation (yellow curve) provides a smoother estimation of the density which matches closely the analytical one (red), see right-hand side panels for example, however, the empirical distribution is step-like and therefore not assessed as being normal. Second, if the statistic at a lag is close to 0 or close to its maximum value ($n_X n_Y$), the integer-valued nature of the statistic will once again prevent the empirical distribution to be approximated as normal. For example, the distribution of the statistic in the first plot in panel B is not symmetric because of its lower bound. This was previously noted by~\cite{palm1988significance}. More generally, this issue is illustrated by the top row of Figure \ref{fig:normality_heat_map} whereby estimates for low rate processes require significantly longer durations for the empirical distributions to become approximately normal. 

\noindent The asymptotic convergence of the statistic around its expected value justifies the construction and use of a Z-score to measure deviation from the assumption of independence. For two time series $X$ and $Y$ emitting $n_X$ and $n_Y$ events within $\llbracket 1; T\rrbracket$, it is defined as:
\begin{equation}
\label{eq:Z-score}
Z_{X,Y}(\tau) = \frac{|S_{T, \tau}| - \E(|S_{T, \tau}|)}{ \sigma_{\delta}\sqrt{T} }
\end{equation}
where $|S_{T, \tau}|$ is the statistic at lag $\tau$, $\E(|S_{T, \tau}|)$ is the analytical expected value~\eqref{eq:expected_value} and $\sigma_{\delta}$ the analytical standard deviation~\eqref{var} where we are estimating parameters $p_X, p_Y$ with their empirical estimators $\frac{n_X}{T}$ and $\frac{n_Y}{T}$ respectively.  

\subsection{Application to a delayed version of the Common Shock Model}
\label{sub:sens_anal}

\subsubsection{Description of model}
The bivariate Poisson Common Shock Model provides a flexible platform to validate the effectiveness of our statistic in quantifying pairwise coupling between two point processes. Indeed it is often used as a benchmark for bivariate counts~\cite{genest2018new}. In its classic form, and following~\cite{campbell1934poisson}, it  takes two Poisson point processes $X_1 \sim \mathcal{P}(\lambda_{X_1})$ and $X_2 \sim \mathcal{P}(\lambda_{X_2})$ and assumes that each of them can be written as the sum of two independent Poisson processes respectively $Z$ and $Y_1$, and $Z$ and $Y_2$, where $Z \sim \mathcal{P}(\lambda_{Z})$, $Y_1 \sim \mathcal{P}(\lambda_{X_1} - \lambda_Z)$ and $Y_2 \sim \mathcal{P}(\lambda_{X_2} - \lambda_Z)$. Whilst the choice of rates $\lambda_{Z}$, $\lambda_{Y_1}$ and $\lambda_{Y_2}$ enables to adjust the strength of the interaction, in this form, the model only allows for instantaneous interaction. Here, we use the delayed version ~\cite{cox1972multivariate} such that, instead of $Z$, we consider two processes $Z_1$ and $Z_2$ that are delayed versions of $Z$, with delays  taken from some probability distribution function $F_i(.), i \in \{1,2\}$ (see Figure~\ref{fig:DCMS} for a graphical representation). This process has a number of possible interpretations but a very prominent one in the neurophysiology literature is the notion of \textit{common drive}~\cite{sears1976short}. Since Z is not actually observed and we are not interested in recovering it, without loss of generality, it is possible to consider that only one of the process, e.g. $Z_2$, is a delayed version of $Z$, the other, $Z_1$ being equal to Z. We thus denote $Z^{\star}$ the delayed version of $Z$. To model the delay we used a normal distribution of parameters $\mu_{\delta}$ and $\sigma_{\delta}$, allowing us to control both average lag between events from $Z$ and $Z^{\star}$ and lag jitter. 

\begin{figure}[h]
   \centering \includegraphics[width=0.5\linewidth]{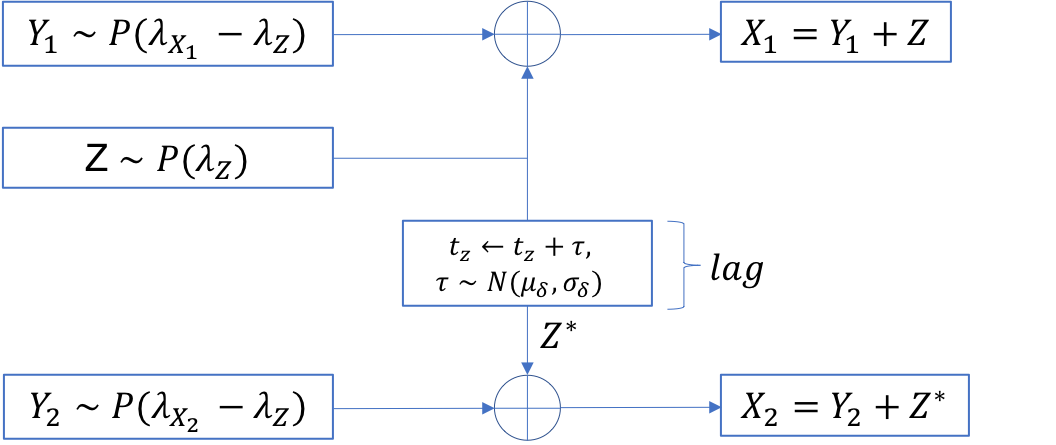}
   \caption{The Delayed Common Shock Model}
   \label{fig:DCMS}
\end{figure}

\subsubsection{Brief validation of the statistic}
\label{sub:lag_analysis}
Since previous results from Section~\ref{sub:sampledPP} showed excellent agreement between empirical and analytical estimates of the expected value and standard deviation of the statistic for independent Poisson processes, application of the Z-score to the delayed common shock model in the no-interaction case (i.e., $\lambda_z=0$) should lead to a Z-score taking an expected value of 0 and a standard deviation of 1. This behaviour is confirmed by the left panel of Figure~\ref{fig:ex_zscore} with consistent empirical estimates (expected value and standard deviation) across all lag values.\newline 

As representative example of the delayed interaction case, we considered the case where all 3 processes had the same rate $\lambda=0.01$, an average lag $\mu_{\delta} = 100$ and jitter $\sigma_{\delta}=10$ (chosen to be significantly smaller than the average lag, as expected in most real-world scenarios). The behaviour of the Z-score is shown in right panel of Figure~\ref{fig:ex_zscore}. At first sight, the interpretation of this behaviour is not straightforward. First, it is observed that whilst the Z-score for value of lags $\tau \ll \mu_{\delta}$ is close to zero, it is not zero. This is a direct implication of the delay introduced by the model. The analytical expected value of the statistic provides the expected number of co-incident events at each lag conditional to the processes emitting a given number of events over the time horizon T. With delayed interaction, and depending on the relationship between the delay and the rate of the processes, the actual number of co-incident events observed for lags less than the delay will diverge from that expected at random. In this case, with each process only emitting (on average) 10 events (i.e., 1 every 100 time steps), a delay of 100 is likely to reduce the likelihood of co-occurrence at very small lags. The Z-score then shows a steady increase peaking at around $\mu_{\delta} + 1.5\sigma_{\delta}$ before steadily decreasing. The fact that the Z-score does not peak at the average lag but later is a natural result of the cumulated nature of our statistic. To illustrate this, we calculated the derivative of the Z-score (orange line) as a proxy for the  cross-correlation at each lag and superimposed the distribution of delays used in the model (green line). As the lag approaches the mean lag, the number of co-incident events begins to exceed that expected at random and therefore, the Z-score increases. The rate of increase is maximum where the distribution of delays peaks. As the lag increases from the mean delay, the ratio between the number of observed co-incident events to the number of co-incident events expected at random reduces and with it the Z-score. Once again, because the Z-score is a cumulative statistic (measuring the number of co-incident events up to the lag considered), the statistic will not return to 0. In other words, the Z-score at high lags takes into account the fact there has been substantial interaction at smaller lags. There are two implications to this. First, the Z-score will not show great sensitivity to interactions that have large (in relation to the duration of the record) average lags. Second, a positive value of the Z-score at a given lag does not indicate that significant interaction is actually taking place at that lag. Our results are based on calculating the statistic at each of all lags considered. Note that the alternative approach of extracting the peak value of the statistic over those lags for each realisation would not permit comparison to a $\mathcal{N}(0,1)$ distribution as noted in~\cite{kramer2009network} for cross-correlation analysis. For such comparison to be made, one needs to look at each lag. However, given the cumulated nature of the Z-score, any substantial interaction will lead to significant Z-score values over a wide range of lags (greater than the mean delay) such that any prior knowledge and/or hypothesis as to the nature of the delay (e.g., conduction velocity in neurophysiology) could be usefully exploited. 

\begin{figure}[h]
\includegraphics[width=1.0\linewidth]{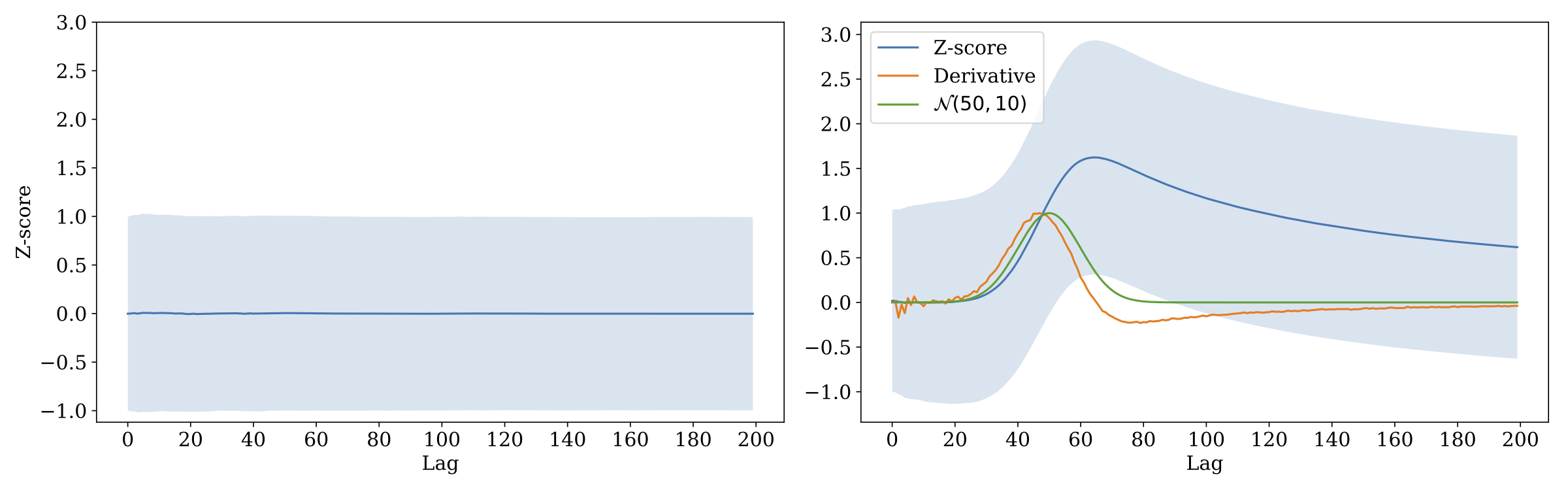}
\caption{Estimated expected value of the Z-score over 10,000 trials as a function of the lag in absence (left) and presence (right) of delayed interaction. We used T=10,000, $\lambda_{y1} = \lambda_{y2} =  \frac{100}{T}$ and $\lambda_{z} =0$ for the left panel; T=10,000, $\lambda_{y1} = \lambda_{y2} = \lambda_{z} =  \frac{100}{T}$, $\mu_{\delta} = 50$ and $\sigma_{\delta}=10$ for the right panel.} 
\label{fig:ex_zscore}
\end{figure}

\section{Discussion}
In this paper, we have proposed a statistic that enables the robust assessment of the presence of (instantaneous or delayed) statistical dependence between two point processes. This statistic extends rigorous work on Pearson's correlation~\cite{johnson1976analysis} by providing rigorous results when considering interaction over multiple lags. The identification of statistical dependence is based on an exact derivation of the number (and standard deviation) of co-incident events expected to be observed over a time period given two discrete independent point processes. This statistic depends only on parameters that can be readily estimated by counting the number of events observed over a given period of time and is valid for every value of those parameters, even extreme cases (e.g., only one event over the entire record), unlike its frequency-domain counterpart, coherence. A central limit theorem for this number was derived that demonstrates convergence of the distribution of the statistic to a normal distribution of known parameters and permits the construction of a Z-score quantifying the likelihood of a given number of co-incident events in a time period being observed if the processes were independent.\newline

\noindent Agreement between analytical and empirical values was verified and the statistic was shown to be well behaved even when departing from the assumptions of the model. For example, it was shown to be robust to sampling (in most reasonable cases) as well as when auto-correlated point processes were considered (within limit). \newline

\noindent The statistic has two main limitations. First, from an interpretation viewpoint, it is somewhat less intuitive than other methods due to its cumulative nature. A significant Z-score at a particular lag does not imply significant underlying dependence at that lag, merely, that there was dependence at one or many shorter lags. This means some care must be taken when selecting the lag(s) over which to make an inference. Nevertheless, we have shown that in some cases, it will be possible to make inference about the nature of the interaction through differencing of the method. Second, in its basic form, and like most such techniques, the statistic assumes stationarity of the signals. Adaptations of the statistic to handle non-stationarity (e.g., via optimal filtering or windowing) is left for future work.

\bibliographystyle{plain}
\bibliography{biblio}

\begin{thebibliography}{10}

\bibitem{arnett1981cross}
D~Arnett and TE~Spraker.
\newblock Cross-correlation analysis of the maintained discharge of rabbit
  retinal ganglion cells.
\newblock {\em The Journal of physiology}, 317(1):29--47, 1981.

\bibitem{bassett_adaptive_2006}
D.~S. Bassett, A.~Meyer-Lindenberg, S.~Achard, T.~Duke, and E.~Bullmore.
\newblock Adaptive reconfiguration of fractal small-world human brain
  functional networks.
\newblock {\em Proceedings of the National Academy of Sciences},
  103(51):19518--19523, December 2006.

\bibitem{blinowska2011review}
Katarzyna~J Blinowska.
\newblock Review of the methods of determination of directed connectivity from
  multichannel data.
\newblock {\em Medical \& biological engineering \& computing}, 49(5):521--529,
  2011.

\bibitem{brown2004multiple}
Emery~N Brown, Robert~E Kass, and Partha~P Mitra.
\newblock Multiple neural spike train data analysis: state-of-the-art and
  future challenges.
\newblock {\em Nature neuroscience}, 7(5):456, 2004.

\bibitem{buckner2009cortical}
Randy~L Buckner, Jorge Sepulcre, Tanveer Talukdar, Fenna~M Krienen, Hesheng
  Liu, Trey Hedden, Jessica~R Andrews-Hanna, Reisa~A Sperling, and Keith~A
  Johnson.
\newblock Cortical hubs revealed by intrinsic functional connectivity: mapping,
  assessment of stability, and relation to alzheimer's disease.
\newblock {\em Journal of neuroscience}, 29(6):1860--1873, 2009.

\bibitem{campbell1934poisson}
JT~Campbell.
\newblock The poisson correlation function.
\newblock {\em Proceedings of the Edinburgh Mathematical Society}, 4(1):18--26,
  1934.

\bibitem{chen2011statistical}
Zhe Chen, David~F Putrino, Soumya Ghosh, Riccardo Barbieri, and Emery~N Brown.
\newblock Statistical inference for assessing functional connectivity of
  neuronal ensembles with sparse spiking data.
\newblock {\em IEEE transactions on neural systems and rehabilitation
  engineering}, 19(2):121--135, 2011.

\bibitem{cox1972multivariate}
David~R Cox and Peter Adrian~Walter Lewis.
\newblock Multivariate point processes.
\newblock In {\em Proc. 6th Berkeley Symp. Math. Statist. Prob}, volume~3,
  pages 401--448, 1972.

\bibitem{eggermont1992neural}
JJ~Eggermont.
\newblock Neural interaction in cat primary auditory cortex. dependence on
  recording depth, electrode separation, and age.
\newblock {\em Journal of Neurophysiology}, 68(4):1216--1228, 1992.

\bibitem{eguiluz2005scale}
Victor~M Eguiluz, Dante~R Chialvo, Guillermo~A Cecchi, Marwan Baliki, and
  A~Vania Apkarian.
\newblock Scale-free brain functional networks.
\newblock {\em Physical review letters}, 94(1):018102, 2005.

\bibitem{fisher1921probable}
Ronald~A Fisher.
\newblock On the probable error of a coefficient of correlation deduced from a
  small sample.
\newblock {\em Metron}, 1:3--32, 1921.

\bibitem{genest2018new}
Christian Genest, Mhamed Mesfioui, and Juliana Schulz.
\newblock A new bivariate poisson common shock model covering all possible
  degrees of dependence.
\newblock {\em Statistics \& Probability Letters}, 140:202--209, 2018.

\bibitem{halliday_nonparametric_2015}
David~M. Halliday.
\newblock Nonparametric directionality measures for time series and point
  process data.
\newblock {\em J. Integr. Neurosci.}, 14(02):253--277, May 2015.

\bibitem{halliday1995framework}
DM~Halliday, JR~Rosenberg, AM~Amjad, P~Breeze, BA~Conway, and SF~Farmer.
\newblock A framework for the analysis of mixed time series/point process
  data-theory and application to the study of physiological tremor, single
  motor unit discharges and electromyograms.
\newblock {\em Progress in biophysics and molecular biology}, 64(2):237, 1995.

\bibitem{hayasaka_comparison_2010}
Satoru Hayasaka and Paul~J. Laurienti.
\newblock Comparison of characteristics between region-and voxel-based network
  analyses in resting-state {fMRI} data.
\newblock {\em NeuroImage}, 50(2):499--508, April 2010.

\bibitem{jarvis2001sampling}
MR~Jarvis and PP~Mitra.
\newblock Sampling properties of the spectrum and coherency of sequences of
  action potentials.
\newblock {\em Neural Computation}, 13(4):717--749, 2001.

\bibitem{johnson1976analysis}
DH~Johnson and NY~Kiang.
\newblock Analysis of discharges recorded simultaneously from pairs of auditory
  nerve fibers.
\newblock {\em Biophysical journal}, 16(7):719, 1976.

\bibitem{kramer2009network}
Mark~A Kramer, Uri~T Eden, Sydney~S Cash, and Eric~D Kolaczyk.
\newblock Network inference with confidence from multivariate time series.
\newblock {\em Physical Review E}, 79(6):061916, 2009.

\bibitem{kwak2010twitter}
Haewoon Kwak, Changhyun Lee, Hosung Park, and Sue Moon.
\newblock What is twitter, a social network or a news media?
\newblock In {\em Proceedings of the 19th international conference on World
  wide web}, pages 591--600. AcM, 2010.

\bibitem{mckenzie_simple_1985}
Ed~McKenzie.
\newblock Some {Simple} {Models} for {Discrete} {Variate} {Time} {Series}1.
\newblock {\em JAWRA Journal of the American Water Resources Association},
  21(4):645--650, 1985.

\bibitem{annet-messager}
Antoine Messager, George Parisis, Robert Harper, Philip Tee, Istvan~Z. Kiss,
  and Luc Berthouze.
\newblock Network events in a large commercial network: What can we learn?
\newblock In {\em Proc. of IEEE/IFIP NOMS AnNet}, 2018.

\bibitem{muller2008estimating}
M~M{\"u}ller, G~Baier, C~Rummel, and K~Schindler.
\newblock Estimating the strength of genuine and random correlations in
  non-stationary multivariate time series.
\newblock {\em EPL (Europhysics Letters)}, 84(1):10009, 2008.

\bibitem{netoff2004analytical}
Theoden~I Netoff, Louis~M Pecora, and Steven~J Schiff.
\newblock Analytical coupling detection in the presence of noise and
  nonlinearity.
\newblock {\em Physical Review E}, 69(1):017201, 2004.

\bibitem{newman_spread_2002}
M.~E.~J. Newman.
\newblock Spread of epidemic disease on networks.
\newblock {\em Phys. Rev. E}, 66(1):016128, July 2002.

\bibitem{newman2018networks}
Mark Newman.
\newblock {\em Networks}.
\newblock Oxford university press, 2018.

\bibitem{palm1988significance}
G~Palm, AMHJ Aertsen, and GL~Gerstein.
\newblock On the significance of correlations among neuronal spike trains.
\newblock {\em Biological cybernetics}, 59(1):1--11, 1988.

\bibitem{pascual2007instantaneous}
Roberto~D Pascual-Marqui.
\newblock Instantaneous and lagged measurements of linear and nonlinear
  dependence between groups of multivariate time series: frequency
  decomposition.
\newblock {\em arXiv preprint arXiv:0711.1455}, 2007.

\bibitem{pipa2013impact}
Gordon Pipa, Sonja Gr{\"u}n, and Carl Van~Vreeswijk.
\newblock Impact of spike train autostructure on probability distribution of
  joint spike events.
\newblock {\em Neural Computation}, 25(5):1123--1163, 2013.

\bibitem{quaglio2018methods}
Pietro Quaglio, Vahid Rostami, Emiliano Torre, and Sonja Gr{\"u}n.
\newblock Methods for identification of spike patterns in massively parallel
  spike trains.
\newblock {\em Biological cybernetics}, 112(1-2):57--80, 2018.

\bibitem{rubinov2010complex}
Mikail Rubinov and Olaf Sporns.
\newblock Complex network measures of brain connectivity: uses and
  interpretations.
\newblock {\em Neuroimage}, 52(3):1059--1069, 2010.

\bibitem{rummel2010analyzing}
Christian Rummel, Markus M{\"u}ller, Gerold Baier, Fr{\'e}d{\'e}rique Amor, and
  Kaspar Schindler.
\newblock Analyzing spatio-temporal patterns of genuine cross-correlations.
\newblock {\em Journal of neuroscience methods}, 191(1):94--100, 2010.

\bibitem{sears1976short}
TA~Sears and D~Stagg.
\newblock Short-term synchronization of intercostal motoneurone activity.
\newblock {\em The Journal of physiology}, 263(3):357--381, 1976.

\bibitem{shao1987normalized}
Xuesi Shao and Peixi Chen.
\newblock Normalized auto-and cross-covariance functions for neuronal spike
  train analysis.
\newblock {\em International journal of neuroscience}, 34(1-2):85--95, 1987.

\bibitem{smith2011network}
Stephen~M Smith, Karla~L Miller, Gholamreza Salimi-Khorshidi, Matthew Webster,
  Christian~F Beckmann, Thomas~E Nichols, Joseph~D Ramsey, and Mark~W Woolrich.
\newblock Network modelling methods for fmri.
\newblock {\em Neuroimage}, 54(2):875--891, 2011.

\bibitem{sporns2010networks}
Olaf Sporns.
\newblock {\em Networks of the Brain}.
\newblock MIT press, 2010.

\bibitem{stevenson2009bayesian}
Ian~H Stevenson, James~M Rebesco, Nicholas~G Hatsopoulos, Zach Haga, Lee~E
  Miller, and Konrad~P Kording.
\newblock Bayesian inference of functional connectivity and network structure
  from spikes.
\newblock {\em IEEE Transactions on Neural Systems and Rehabilitation
  Engineering}, 17(3):203--213, 2009.

\bibitem{wang2009parcellation}
Jinhui Wang, Liang Wang, Yufeng Zang, Hong Yang, Hehan Tang, Qiyong Gong, Zhang
  Chen, Chaozhe Zhu, and Yong He.
\newblock Parcellation-dependent small-world brain functional networks: a
  resting-state fmri study.
\newblock {\em Human brain mapping}, 30(5):1511--1523, 2009.

\bibitem{zandvakili2015coordinated}
Amin Zandvakili and Adam Kohn.
\newblock Coordinated neuronal activity enhances corticocortical communication.
\newblock {\em Neuron}, 87(4):827--839, 2015.

\end{thebibliography}

\appendix
\section{The case of Bins for Poisson processes}
\label{app:binning}

Suppose that $N_X$, $N_Y$ are independent Poisson processes with rates $\lambda_X$, and $\lambda_Y$ respectively. 
We will discretise these processes using bins of size $b > 0$, with a time horizon $T$. \newline

\noindent Partition the time interval $[0,T]$ into disjoint intervals of size $b$. In other words, interval $k$ (denoted by $I_k$) is the interval $[(k-1)b, kb )$, for $1 \le k  \le [Tb^{-1}]$. 
The Poisson processes are independent in each interval, so we define a sequence of independent Bernoulli variables 
\be
X_i = 1\!\!1\{ N_X(ib) - N_X((i-1)b) \ge 1 \}, \quad \text{ and } \quad  Y_j = 1\!\!1\{ N_Y(ib) - N_Y((i-1)b) \ge 1\}
\ee
Their respective probabilities of success are 
\be \label{eq:poi2ber}
p_X = 1 - e^{-\lambda_X b}  \quad \text{ and } \quad p_Y = 1 - e^{-\lambda_Y b}. 
\ee
Therefore, using the bin size $b$, we reduce this to our discrete model, with discrete time horizon $T_d(b) = [Tb^{-1}]$ and success probabilities for the Bernoulli variables.~\label{eq:poi2ber} The lag value $\delta$ from the continuous model is also scaled by the bin size; since it needs to be an integer, by convention we take it to be 
$ \delta_{d}(b) = [\delta b^{-1}]$. \newline

\noindent Using this information, the expected number of marks will be $\E(|\widetilde S_{T_{d}(b), \delta_d(b)}|)$ and the variance needed for the CLT will be of leading order $\sigma_{\delta_d(b)}^2 T_d(b)$. 
Using these theoretical results, one can quantify the estimation error induced by binning according to $b$, since we will be ignoring all events but 1 that happen in the same sub-interval $I_k$.\newline

\noindent Suppose for convenience that $\lambda_X \ge \lambda_Y$. 
The probability $p_{\lambda_X, b}$ of seeing 2 or more Poisson points for $N_X$ in $I_k$ is 
\begin{align*}
p_{\lambda_X, b} &= 1 -  e^{-\lambda_X b} (1 + \lambda_X b)\le 1- (1 - \lambda_X b) (1 + \lambda b) = (\lambda_X b)^2.
\end{align*}
Therefore, if $b \ll \lambda_X^{-1}$ then this probability will be small. \newline

\noindent Furthermore, if $\lambda_X b \le C T^{-\alpha} $ for some $\alpha > 1/2$, then a Taylor expansion of the expected value of points lost because of binning shows that $\E(N_L) \ll \sqrt{T}$ which is not enough to alter the leading order of $\sigma_\delta^2 \sqrt T$ and $\E(| \widetilde S_{T, \delta}|)$. In this case the Poisson model and the discrete model obtained by decreasing bin sizes will coincide.

\end{document}